\documentclass[graybox]{svmult}

\usepackage{mathptmx}		% selects Times Roman as basic font
\usepackage{helvet}			% selects Helvetica as sans-serif font
\usepackage{courier}		% selects Courier as typewriter font
\usepackage{type1cm}		% activate if the above 3 fonts are
\usepackage{float}

\usepackage{makeidx}		% allows index generation
\usepackage{graphicx}		% standard LaTeX graphics tool
\usepackage{multicol}		% used for the two-column index
\usepackage[bottom]{footmisc}	% places footnotes at page bottom

\newcommand{\abs}[1]{\lvert#1\rvert}
\newcommand{\gwt}{\widetilde{g}}

\spnewtheorem*{Poisson-tail-conjecture}{Poisson Tail Conjecture}{\normalfont\bfseries}{\it}
\spnewtheorem*{prime-k-tuples-conjecture}{Hardy--Littlewood Prime $k$-tuples Conjecture}{\normalfont\bfseries}{\it}
\spnewtheorem*{strong-prime-k-tuples-conjecture}{Conjecture Strong HL}{\normalfont\bfseries}{\it}
\spnewtheorem*{Montgomery-theorem}{Montgomery's Theorem}{\normalfont\bfseries}{\it}
\spnewtheorem*{Montgomery-conjecture}{Montgomery's Conjecture}{\normalfont\bfseries}{\it}
\spnewtheorem*{eqivalence-theorem}{Equivalence Theorem}{\normalfont\bfseries}{\it}
\spnewtheorem*{bounded-F-alpha-conjecture}{Bounded $F (\alpha)$ Conjecture}{\normalfont\bfseries}{\it}
\spnewtheorem*{strong-Hardy-Littlewood-conjecture}{Strong Hardy--Littlewood Prime $k$-Tuples Conjecture}{\normalfont\bfseries}{\it}
\spnewtheorem*{notation}{Notation}{\normalfont\bfseries}{\normalfont}
\spnewtheorem*{acknowledgment}{Acknowledgement}{\normalfont\bfseries}{\normalfont}
\usepackage{amssymb}
\usepackage{amsmath}

\makeindex	% used for the subject index
\begin{document}

\title*{Distribution of Large Gaps Between Primes}
\titlerunning{Distribution of Large Gaps Between Primes}
\author{Scott Funkhouser, Daniel A. Goldston, and Andrew H. Ledoan}
\authorrunning{S. Funkhouser, D.A. Goldston, and A.H. Ledoan}
\institute{Scott Funkhouser \at Space and Naval Warfare Systems Center Atlantic,
 One Innovation Drive,
 North Charleston, South Carolina, 29419-9022,
 \email{scott.funkhouser@gmail.com}
\and Daniel A. Goldston \at Department of Mathematics and Statistics,
 San Jos\'{e} State University,
 315 MacQuarrie Hall,
 One Washington Square,
 San Jos\'{e}, California 95192-0103,
 \email{daniel.goldston@sjsu.edu}
\and Andrew H. Ledoan \at Department of Mathematics,
 University of Tennessee at Chattanooga,
 415 EMCS Building (Mail Stop 6956),
 615 McCallie Avenue,
 Chattanooga, Tennessee 37403-2598,
 \email{andrew-ledoan@utc.edu}}

\maketitle

%Please use the 'starred' version of the new Springer \texttt{abstract} command
% for typesetting the text of the online abstracts (cf. source file of this chapter
% template \texttt{abstract}) and include them with the source files of your
% manuscript. Use the plain \texttt{abstract} command if the abstract is also
% to appear in the printed version of the book.

\abstract{We survey some past conditional results on the distribution of large differences between consecutive primes and examine how the Hardy--Littlewood prime $k$-tuples conjecture can be applied to this question.}


\section{Introduction}

The distribution of gaps between consecutive primes around their average spacing is expected to be distributed in a Poisson distribution. Thus, while at first glance the  sequence of gaps appears random and irregular, we expect that they follows a very regular and well-behaved probability distribution. However, as we move to the distribution of larger than average gaps we expect to find increasing irregularity, especially as we reach the limiting size for these gaps. One would guess that near this maximal gap size any distribution will be exceedingly irregular. However, at present the available theoretical tools and well-accepted conjectures do not provide any widely believed standard model for large gaps. 

Define 
\begin{equation} \label{N (x, H)}
N (x, H)
 = \sum_{\substack{p_{n + 1} \leq x \\ p_{n + 1} - p_{n} \geq H}} 1
\end{equation}
and the weighted counting function
\begin{equation} \label{S (x, H)}
S (x, H)
 = \sum_{\substack{p_{n + 1} \leq x \\ p_{n + 1} - p_{n} \geq H}}  (p_{n + 1} - p_{n}) .
\end{equation}
Gallagher \cite{Gallagher1976} showed that a uniform version of the prime $k$-tuples conjecture of Hardy and Littlewood  implies that the primes are distributed in a Poisson distribution around their average.  
In Section 2 we will discuss how this result implies that, for fixed $\lambda>0$, \begin{equation} \label{Npoisson}
N (x, \lambda \log x)
 \sim e^{-\lambda} \frac{x}{\log x}, \qquad \text{as} \ \ x \to \infty,
\end{equation} 
and
\begin{equation} \label{Spoisson}
S (x, \lambda \log x)
 \sim (1+ \lambda) e^{-\lambda} x, \qquad \text{as} \ \ x \to \infty.
\end{equation} 
These results are for fixed $\lambda$, but we are interested in larger gaps. One approach is to assume that the Hardy--Littlewood conjectures hold for primes up to $x$ and for all $k$-tuples, where $k \leq f (x) \to \infty$ for some specified function $f (x)$, together with some strong error term. 
The Hardy--Littlewood conjectures will certainly fail when $k \asymp \log x$ and the error terms are often of size greater than $x^{1 / 2}$, and therefore this approach has definite limitations. There are also obstacles in  applying these conjectures to $N(x,H)$ and $S(x,H)$.  However, if we ignore these issues and consider this approach as only heuristic, then the following conjecture seems reasonable.

\begin{Poisson-tail-conjecture}
For any $\epsilon > 0$ and $1 \leq H \leq \log^{2 - \epsilon} x$, we have
\begin{equation} \label{Hope}
N (x, H)
 \asymp  e^{-H / \log x} \frac{x}{\log x} \quad \text{and} \quad
S (x, H)
 \asymp \left(1 + \frac{H}{\log x}\right) e^{-H / \log x} x.
\end{equation}
For $H > \log^{2 + \epsilon} x$, we have
\begin{equation} \label{largestgap}
N (x, H)
 = S (x, H)
 =0.
\end{equation}
\end{Poisson-tail-conjecture}

Here, $f (x) \asymp g (x)$ means $f (x) \ll g (x)$ and $g (x) \ll f (x)$. In the critical range $\log^{2 - \epsilon} x \leq H \leq  \log^{2 + \epsilon} x$, we have nothing to contribute. Other authors have made stronger conjectures than \eqref{largestgap}. In 1935 Cram{\'e}r \cite{Cramer35}  conjectured that
\begin{equation}
\limsup_{p_{n} \to \infty} \frac{p_{n + 1} - p_{n}}{\log^{2} p_{n}}
 = 1,
\end{equation}
while  in 1995 Granville \cite{Granville1995} conjectured that Cram{\'e}r's conjecture is false and that
\begin{equation}
\limsup_{p_{n} \to \infty} \frac{p_{n + 1} - p_{n}}{\log^{2} p_{n}}
 \geq 2 e^{-\gamma}
 = 1.12292\ldots,
\end{equation}
based on a Cram{\'e}r model modified to include divisibility by small primes. Our conjecture is much weaker and only implies that for any $\delta > 0$, there are prime gaps of size $> \log^{2 - \delta} p_{n}$, and there are no prime gaps of size $> \log^{2 + \delta} p_{n}$. It should be mentioned however that the same modification Granville used in the Cram{\'e}r model was exploited by Maier \cite{Maier1985} to prove there is no asymptotic formula for the number of primes in intervals $(x, x+ \log^{C} x]$, for any given positive number $C$. This result of Maier demonstrates that all of these conjectures on large gaps are far from certain. 

Our purpose in this paper is to describe earlier conditional work on large gaps between primes. The earliest such work assumes the Riemann Hypothesis and is mainly due to Cram{\'e}r \cite{Cramer36}  and to Selberg \cite{Selberg1942}. Later, after Montgomery's work on the pair correlation of zeros of the Riemann zeta-function \cite{Montgomery1972}, some of these Riemann Hypothesis results were slightly improved by Gallagher and Mueller \cite{GallagherMueller1978}, by Mueller \cite{Mueller1981}, and by Heath-Brown \cite{Heath-Brown82} assuming a pair correlation conjecture. All of these results can be obtained from estimating the second moment (or variance) of the number of primes in short intervals. As an application of these second moment results, one can obtained conditionally nearly optimal upper bounds  on the sum
\begin{equation} \label{Csum}
\mathcal{C} (x)
 = \sum_{p_{n + 1} \leq x} (p_{n + 1} - p_{n})^{2},
\end{equation}
and also the closely related and slightly simpler sum
\begin{equation} \label{Ssum}
\mathcal{S} (x)
 = \sum_{p_{n+1} \leq x} \frac{(p_{n + 1} - p_{n})^{2}}{p_{n + 1}}.
\end{equation}
As far as we know, $\mathcal{C} (x)$ was first bounded on the Riemann Hypothesis by Cram{\'e}r \cite{Cramer36} in 1936, while $\mathcal{S} (x)$ was studied by Selberg \cite{Selberg1942} in 1942. 

To obtain stronger results, we return to the Hardy--Littlewood conjectures. The conjecture for pairs (or 2-tuples) with a strong error term is well-known to provide the same estimates for the second moment for primes in short intervals as that obtained by assuming the Riemann Hypothesis and the Pair Correlation Conjectures. In 2004 Montgomery and Soundararajan \cite{MontgomerySoundararajan2004} were able to extend this method to give asymptotic formulas for the $2k$-th moments for the primes in short intervals, assuming the Hardy--Littlewood conjecture for tuples of size $\leq 2 k$ with a strong error term in the conjecture. At present, this approach is the most promising direction towards connecting results on large gap problems to a well-established, if extremely difficult, conjecture on primes. In Section 8 of this paper we use a fourth moment result to nearly resolve the conjectured asymptotic formulas for $\mathcal{C} (x)$ and $\mathcal{S} (x)$. 

\begin{notation}
We always assume that $k, m$, and $n$ are integers. We denote the $n$-th prime by $p_{n}$, and $p$ will always denote a prime. By $\epsilon$ we mean any sufficiently small positive real number. 
\end{notation}

\begin{acknowledgment}
The authors wish to express their sincere gratitude and appreciation to the anonymous referee for carefully reading the original version of this paper and for making a number of very helpful comments and suggestions.
\end{acknowledgment}
 
\section{Gallagher's Theorem and the Poisson Distribution of Primes}

We first introduce the Hardy--Littlewood prime $k$-tuples conjecture in the form used by Gallager. Let $\mathcal{H}_{k} = \{h_{1}, \ldots, h_{k}\}$ be a set of $k$ distinct nonnegative integers. Let $\pi (x; \mathcal{H}_{k})$ denote the number of positive integers $n \leq x$ for which $n + h_{1}, \ldots, n + h_{k}$ are simultaneously primes. Then the simplest form of the Hardy--Littlewood prime $k$-tuples conjecture \cite{HardyLittlewood1922} may be stated as follows.

Let
\begin{equation} \label{singular}
\mathfrak{S} (\mathcal{H}_{k})
 = \prod_{p} \left(1 - \frac{1}{p}\right)^{-k} \left(1 - \frac{\nu_{\mathcal{H}_{k}} (p)}{p}\right),
\end{equation}
where
$\nu_{\mathcal{H}_{k}} (p)$ denotes the number of distinct residue classes modulo $p$ occupied by the elements of $\mathcal{H}_{k}$. Note that, in particular, if $\nu_{\mathcal{H}_{k}} (p) = p$ for some prime number $p$, then $\mathfrak{S} (\mathcal{H}_{k}) = 0$. However, if $\nu_{\mathcal{H}_{k}} (p) < p$ for all prime numbers $p$, then  $\mathfrak{S} (\mathcal{H}_{k}) \neq 0$ in which case the set $\mathcal{H}_{k}$ is called {\it admissible}.

\begin{prime-k-tuples-conjecture}
For each fixed integer $k \geq 2$ and admissible set $\mathcal{H}_{k}$, we have
\begin{equation} \label{hlconj}
\pi (x; \mathcal{H}_{k})
 = \mathfrak{S} (\mathcal{H}_{k}) \frac{x}{\log^{k} x} (1 + o_{k} (1)),
\end{equation}
uniformly for $\mathcal{H}_{k} \subset [1, h]$, where $h \sim \lambda \log x$ as $x \to \infty$ and $\lambda$ is a positive constant.
\end{prime-k-tuples-conjecture}

If $\mathcal{H}_{k}$ is not admissible, then there is a fixed prime $p$ that always divides at least one of the $k$ numbers $n + h_{i}$, with $1 \leq i \leq k$, and hence 
\begin{equation}\label{notadmissible}  \pi(x; \mathcal{H}_k) \leq k  \qquad \text{if} \ \ \mathcal{H}_k \  \ \text{is not admissible.}  \end{equation}
While a proof of \eqref{hlconj} appears beyond our current state of knowledge, we do know by sieve methods (see \cite{HalberstamRichert1974}) the useful upper bound
\begin{equation} \label{HLsievebound}
\pi (x; \mathcal{H}_{k})
 \ll_{k} \mathfrak{S} (\mathcal{H}_{k}) \frac{x}{\log^{k} x}.
\end{equation}

\begin{theorem}[\textbf{Gallagher}] \label{theorem-1}
Let $P_{k} (N, h)$ denote the number of positive integers $n \leq N$ for which the interval $(n, n + h]$ contains exactly $k$ primes. Assuming the Hardy--Littlewood prime $k$-tuples conjecture, we have
\begin{equation*}
P_{k} (N, h)
 \sim \frac{e^{-\lambda} \lambda^{k}}{k!} N, \quad \text{for} \ \ h \sim \lambda \log N \ \ \text{as} \ \ N \to \infty.
\end{equation*}
\end{theorem}
The Poisson distribution of primes manifests itself in Gallagher's proof through the fact that the singular series is on average asymptotic to 1 when averaged over all tuples. Gallagher \cite{Gallagher1976} proved  that, as $h \to \infty$,
\begin{equation} \label{GalSingAv}
\sum_{\substack{1 \leq h_{1}, \ldots, h_{k} \leq h \\ h_{1}, \ldots, h_{k} \ \text{distinct}}}
 \mathfrak{S} (\mathcal{H}_{k})
 = h^{k} + O(h^{k - 1 / 2 + \epsilon}),
\end{equation}
for each fixed $k \geq 2$.

\begin{proof} We give Gallagher's proof. For $k$ a positive integer, the $k$-th moment for the number of primes in the interval $(n, n + h]$ is
\begin{equation*}
M_{k} (N)
 = \sum_{n \leq N} (\pi (n + h) - \pi (n))^{k}
 = \sum_{n \leq N} \sum_{n < p_{1}, \ldots, p_{k} \leq n + h} 1.
\end{equation*}
We group terms according to the number $r$ of distinct primes among the primes $p_{1}, \ldots, p_{k}$ and obtain
\begin{equation*}
M_{k} (N)
 = \sum_{r = 1}^{k} \left\{\begin{array}{c} k \\ r \end{array}\right\} \sum_{\substack{1 \leq h_{1}, \ldots, h_{r} \leq h \\ h_{1}, \ldots, h_{r} \ \text{distinct}}} \pi (N; \mathcal{H}_{r}),
\end{equation*}
where $\left\{\begin{array}{c} k \\ r \end{array}\right\}$ is used to denote the Stirling number of the second type equal to the number of partitions of a set of $k$ elements into $r$ nonempty subsets.  
By \eqref{hlconj}, \eqref{notadmissible}, and \eqref{GalSingAv}, we have for fixed $k$ and $h \sim \lambda \log N$ as $N \to \infty$,
\begin{equation*}
\begin{split} M_{k} (N)
 &\sim \sum_{r = 1}^{k} \left\{\begin{array}{c} k \\ r \end{array}\right\} \sum_{\substack{1 \leq h_{1}, \ldots, h_{r} \leq h \\ h_{1},\ldots, h_{r} \ \text{distinct}}} \mathfrak{S} (\mathcal{H}_{r}) \frac{N}{(\log N)^{r}} \\
 &\sim \sum_{r = 1}^{k} \left\{\begin{array}{c} k \\ r \end{array}\right\} h^{r}\frac{N}{(\log N)^{r}} \\
 &\sim m_{k} (\lambda) N,
\end{split}
\end{equation*}
where  
\begin{equation} \label{Poissonmoment}
m_{k} (\lambda)
 = \sum_{r = 1}^{k} \left\{\begin{array}{c} k \\ r \end{array} \right\} \lambda^{r},
\end{equation}
which is the $k$-th Poisson moment with expected value $\lambda$. Theorem \ref{theorem-1} now follows from the standard theorems on moments.
\end{proof}

From Theorem \ref{theorem-1}, we now prove \eqref{Npoisson} and \eqref{Spoisson}.
\begin{theorem} \label{theorem-2}
Assuming the Hardy--Littlewood prime $k$-tuples conjecture, then for fixed $\lambda > 0$, and $H \sim \lambda \log x$ as $x \to \infty$,  we have
\begin{equation} \label{thmNpoisson}
N (x, \lambda \log x)
 = \sum_{\substack{p_{n + 1} \leq x \\ p_{n + 1} - p_{n} \geq \lambda \log x}} 1
 \sim e^{-\lambda} \frac{x}{\log x}
\end{equation} 
and
\begin{equation} \label{thmSpoisson}
S (x, \lambda \log x)
 = \sum_{\substack{p_{n + 1} \leq x \\ p_{n + 1} - p_{n} \geq \lambda \log x}} (p_{n + 1} - p_{n})
 \sim (1 + \lambda) e^{-\lambda} x.
\end{equation}
\end{theorem}

\begin{proof} Let  
\begin{equation} \label{S1}
S_{1} (x, H)
 = \sum_{\substack {p_{n} \leq x \\ p_{n + 1} - p_{n} \geq H}} ((p_{n + 1} - p_{n}) - H).
\end{equation}
Taking $k = 0$ in Theorem \ref{theorem-1}, we  have 
\begin{equation*}
P_{0} (N, h)
 \sim e^{-\lambda} N,
\end{equation*}
where $P_{0} (N, h)$ is the number of $j \leq N$ for which the interval $(j, j + h]$ contain no primes. This interval has the same number of primes as the interval $[j + 1, j + \lfloor h\rfloor]$, which contains no primes if and only if  there is an $n$ for which $p_{n} \leq j$ and $p_{n + 1} \geq j + \lfloor h\rfloor + 1$, which can occur if and only if $p_{n + 1} - p_{n} \geq \lfloor h\rfloor + 1$. Hence, in this case, $ p_{n} \leq j \leq p_{n + 1} - \lfloor h\rfloor - 1$, and there are $ p_{n + 1} - p_{n} - \lfloor h\rfloor$ such $j$'s for this $p_{n}$. Thus, we have
\begin{equation*}
\begin{split}
P_{0} (N, h)
 &= \sum_{\substack {p_{n} \leq N \\ p_{n + 1} - p_{n} \geq \lfloor h\rfloor}} (p_{n + 1} - p_{n} - \lfloor h\rfloor ) \\
 &= \sum_{\substack {p_{n} \leq N \\ p_{n + 1} - p_{n} \geq h}} (p_{n + 1} - p_{n} - \lfloor h\rfloor ) \\
 &= S_{1} (N, h) + O \left(\frac{N}{\log N}\right).
\end{split}
\end{equation*}
We conclude that, for $H \sim \lambda \log x$,
\begin{equation} \label{S1asymptotic}
S_{1} (x, H)
 \sim e^{-\lambda} x.
\end{equation}
From \eqref{N (x, H)}, \eqref{S (x, H)}, and \eqref{S1},
\begin{equation*}
S_{1} (x, H)
 = S (x, H) - H N (x, H),
\end{equation*}
so that \eqref{Spoisson} follows from \eqref{Npoisson} and \eqref{S1asymptotic}.  To prove \eqref{Npoisson}, we note that
\begin{equation} \label{S1-Nformula}
S_{1} (x, H)
 = \int_{H}^{\infty} N (x, u) \,du
\end{equation}
and, since $N (x, u)$ is a nonincreasing function of $u$, we have for any $\delta > 0$,
\begin{equation*}
\frac{1}{\delta H} \int_{H}^{(1 + \delta) H} N (x, u) \,du
 \leq N (x, H)
 \leq \frac{1}{\delta H} \int_{(1 - \delta) H}^{H} N (x, u) \,du.
\end{equation*}
Therefore, $N (x, H)$ is bounded between 
\begin{equation*}
\frac{S_{1} (x, H) - S_{1} (x, (1 \pm \delta) H)}{\pm \delta H},
\end{equation*}
which by \eqref{S1asymptotic} is, as $\delta \to 0$, 
\begin{equation*}
\begin{split}
&\sim \left(\frac{e^{-\lambda} - e^{-(1 \pm \delta) \lambda}}{\pm \delta \lambda}\right) \frac{x}{\log x} \\
&\sim \left(\frac{1- e^{\mp \delta \lambda}}{\pm \delta \lambda}\right)e^{-\lambda} \frac{x}{\log x} \\
&\sim (1 + O (\delta \lambda)) e^{-\lambda} \frac{x}{\log x} \\
&\sim e^{-\lambda} \frac{x}{\log x},
\end{split}
\end{equation*}
thus proving \eqref{Npoisson}. 
\end{proof}

There is an alternative approach for proving Theorem \ref{theorem-2} which avoids moments. In \cite{GoldstonLedoan2013} the second two authors proved, using inclusion-exclusion with the Hardy--Littlewood prime $k$-tuples conjecture, that for fixed $\lambda > 0$ and $d \sim \lambda \log x$, 
\begin{equation} \label{eq2}
\mathcal{N} (x, d)
 = \sum_{\substack{p_{n + 1} \leq x \\ p_{n + 1} - p_{n} = d}} 1
 \sim e^{-\lambda} \mathfrak{S} (d) \frac{x}{\log^{2} x},
\end{equation}
where
\begin{equation*}
\mathfrak{S} (d)
= \left\{\begin{array}{ll}
      {\displaystyle 2 C_{2} \prod_{\substack{p \mid d \\ p > 2}} \left(\frac{p - 1}{p - 2}\right)} & \mbox{if $d$ is even,} \\
      0 &  \mbox{if $d$ is odd,}
\end{array}
\right.
\end{equation*}
and
\begin{equation*}
C_{2}
 = \prod_{p > 2} \left( 1 - \frac{1}{(p - 1)^{2}}\right)
 = 0.66016\ldots.
\end{equation*}
Here, $\mathfrak{S} (d)$ is the singular series given in \eqref{singular} when $k = 2$ and $\mathcal{H}_{2} = \{0, d\}$. As a consequence of \eqref{GalSingAv} we have
\begin{equation*}
\sum_{d \leq h} \mathfrak{S} (d)
 \sim h.
\end{equation*}
Hence, with $H \sim \lambda \log x$, by partial summation, 
\begin{equation*}
\begin{split}
\sum_{\substack{p_{n + 1} \leq x \\ p_{n + 1} - p_{n} < H}} 1
 &= \sum_{d < H} \mathcal{N} (x, d) \\
 &\sim \left(\sum_{d < H} e^{-d / \log x} \mathfrak{S} (d)\right) \frac{x}{\log^{2} x} \\
 &\sim \left(\int_{0}^{\lambda \log x} e^{-u / \log x} \,du\right) \frac{x}{\log^{2} x} \\
 &\sim (1 - e^{-\lambda}) \frac{x}{\log x}.
\end{split}
\end{equation*}
Thus,
\begin{equation*}
\begin{split}
N (x, H)
 &= \sum_{p_{n} \leq x} 1 - \sum_{\substack{p_{n + 1} \leq x \\ p_{n + 1} - p_{n} < H}} 1 \\
 &\sim e^{-\lambda} \frac{x}{\log x},
\end{split}
\end{equation*}
which proves \eqref{Npoisson}. The same argument gives 
\begin{equation*}
\begin{split}
S (x, H)
 &= \sum_{p_{n} \leq x} (p_{n + 1} - p_{n}) - \sum_{\substack{p_{n + 1} \leq x \\ p_{n + 1} - p_{n} < H}} (p_{n + 1} - p_{n}) \\
 &\sim x - \sum_{d < H} d \mathcal{N} (x, d) \\
 &\sim x- \left(\int_{0}^{\lambda \log x} u e^{-u / \log x} \,du\right) \frac{x}{\log^{2} x} \\
 &\sim (1 + \lambda  )e^{-\lambda} x,
\end{split}
\end{equation*}
which proves  \eqref{Spoisson}.

\section{Bounding the Number of Large Gaps with Moments}

Moving to larger gaps between primes, we reduce our goal of finding their distribution and only seek to find bounds on their frequency. A simple method for bounding the number of large gaps was introduced by Selberg \cite{Selberg1942}. Let
\begin{equation} \label{thetamoment}
M_{2 k} (x, h)
 = \int_{1}^{x} (\vartheta (y + h) - \vartheta (y) - h)^{2 k} \,dy,
\end{equation}
where
\begin{equation*}
\vartheta (x)
 = \sum_{p \leq x} \log p
\end{equation*}
and $k$ is a positive integer.

\begin{lemma} \label{lemma-1}
For $k \geq 1$ and $H \geq 1$, we have
\begin{equation}
S (x, H)
 = \sum_{\substack{p_{n + 1} \leq x \\ p_{n + 1} - p_{n} \geq H}} (p_{n + 1} - p_{n})
 \ll \left(\frac{2}{H}\right)^{2 k} M_{2 k} \left(x, \frac{H}{2}\right).
\end{equation}
\end{lemma}

\begin{proof} We have $\vartheta (y + h) - \vartheta (y) = 0$ whenever there are consecutive primes $p_{n} \leq y$ and $y + h < p_{n + 1}$, which is when $y \in [p_{n}, p_{n + 1} - h)$ which has length $(p_{n + 1} - p_{n}) - h$. Suppose $p_{n + 1} - p_{n} \geq H$ and take $h = H / 2$. Then
\begin{equation*}
\int_{p_{n}}^{p_{n + 1} - h} (\vartheta (y + h) - \vartheta (y) - h)^{2 k} \,dy
 = h^{2 k} ((p_{n + 1} - p_{n}) - h)
 \geq \frac{1}{2} h^{2 k} (p_{n + 1} - p_{n}).
\end{equation*}
Summing over $p_{n + 1} \leq x$ gives the result.
\end{proof}

There is a slightly different moment often used in this subject, namely 
\begin{equation} \label{psimoment}
m_{2 k} (x, h)
 = \int_{1}^{x} (\psi (y + h) - \psi (y) - h)^{2 k} \,dy,
\end{equation}
where
\begin{equation*}
\psi (x)
 = \sum_{p^{m} \leq x} \log p
 = \sum_{n \leq x} \Lambda (n).
\end{equation*}
The next lemma shows that the two moments are essentially the same size. The error term here saves only a power of $\log x$ over the actual size of these moments, but that is sufficient for our applications.

\begin{lemma} \label{lemma-2}
For $k \geq 1, \ x \geq 2$, and $h \geq 1$, we have
\begin{equation*}
m_{2 k} (x, h)^{1 / 2 k}
 = M_{2 k} (x, h)^{1 / 2 k} + O \left((x h^{k})^{1 / 2 k}\right).
\end{equation*}
\end{lemma}
\begin{proof} We recall Minkowski's inequality for integrals (see \cite{Montgomery2014})
\begin{equation*}
\left(\int_{a}^{b} \abs{f (x) + g (x)}^{p} \,dx\right)^{1 / p}
 \leq \left(\int_{a}^{b} \abs{f (x)}^{p} \,dx \right)^{1 / p} + \left(\int_{a}^{b} \abs{g (x)}^{p} \,dx\right)^{1 / p}.
\end{equation*}
Since 
\begin{equation*}
\psi (y + h) - \psi (y) - h
 = (\vartheta (y + h) - \vartheta (y) - h) + R (y, h),
\end{equation*}
where
\begin{equation*}
R (y, h)
 = \sum_{\substack{y < p^{m} \leq y + h \\ m \geq 2}} \log p,
\end{equation*}
we obtain from Minkowski's inequality
\begin{equation} \label{minkowski}
m_{2 k} (x, h)^{1 / 2 k}
 = M_{2 k} (x, h)^{1 / 2 k} + O\left(\left(\int_{1}^{x} R (y, h)^{2 k} \,dy\right)^{1 / 2 k}\right).
\end{equation}
It remains to estimate $R (y, h)$. The inequality $y < p^{m} \leq y + h$ is equivalent to $\sqrt[m]{y} < p \leq \sqrt[m]{y + h}$. We make use of the inequality $\sqrt[m]{y + h} \leq \sqrt[m]{y} + \sqrt[m]{h}$ when $h$ is large and the inequality $\sqrt[m]{y + h} \leq \sqrt[m]{y} (1 + h / m y)$ when $h$ is small.

We consider first the case when $h \geq x^{\delta}$, for some fixed $\delta > 0$. Using the sieve bound 
\begin{equation*}
\pi (y + H) - \pi (y)
 \ll \frac{H}{\log H},
\end{equation*}
we have for $1 \leq y \leq x$
\begin{equation*}
\begin{split}
R (y, h)
 &\leq \sum_{\substack{\sqrt[m]{y} < p \leq \sqrt[m]{y} + \sqrt[m]{h} \\ m \geq 2}} \log p \\
 &\ll \sum_{2 \leq m \leq \log x} \frac{m \sqrt[m]{h}}{\log h}\log x \\
 &\ll \frac{\sqrt{h} \log x}{\log h} + \frac{\sqrt[3]{h} \log^{3} x}{\log h} \\
 &\ll \sqrt{h}.
\end{split}
\end{equation*}
Substituting this bound into \eqref{minkowski} proves Lemma \ref{lemma-2} in this range. 

Next, we consider the range $1 \leq h \leq x^{\delta}$. Estimating trivially, we have 
\begin{equation*}
R (y, h)
 \leq \sum_{\substack{\sqrt[m]{y} < p \leq \sqrt[m]{y} + h / (m y^{1 - 1 / m}) \\ m \geq 2}} \log p
 \ll \frac{h \log^{2} x}{y^{1 / 2}}.
\end{equation*}
Hence,
\begin{equation*}
\int_{1}^{x} R (y, h)^{2 k} \,dy \ll h^{2 k} (\log x)^{4 k + 1}
 \ll x^{2 k \delta} (\log x)^{4 k + 1}
 \ll x,
\end{equation*}
on taking $\delta = 1 / 4 k$.  This proves Lemma \ref{lemma-2} in this range. 
\end{proof}

\section{Second Moment Results Assuming the Riemann Hypothesis}

In \cite{Selberg1942} Selberg proved\footnote{Selberg also proved an unconditional estimate that we are not concerned with in this paper.} that, assuming the Riemann Hypothesis, for $T \geq 2$,
\begin{equation} \label{Selbergintegral}
\int_{1}^{T^{4}} \left(\vartheta \left(y + \frac{y}{T}\right) - \vartheta (y) - \frac{y}{T}\right)^{2}\frac{\,dy}{y^{2}}
 \ll \frac{\log^{2} T}{T}.
\end{equation}
The left-hand side, here, is a damped second moment for primes in short intervals where the interval length varies as a fixed multiple of where it is located. We will make use of \eqref{Selbergintegral} in Section 5. Most authors use in place of Selberg's second moment either $M_{2} (x, h)$ or $m_{2} (x, h)$. Saffari and Vaughan \cite{SaffariVaughan1977} found a method for going back and forth between moments using fixed intervals $[x, x + h]$ and moments using intervals $(x, x + \delta x]$. (See, also, \cite{GoldstonMontgomery1987}.) The result corresponding to \eqref{Selbergintegral} is, for $1 \leq h \ll x^{3 / 4}$,
\begin{equation} \label{2ndmomentRH}
M_{2} (x, h)
 \ll h x \log^{2} x.
\end{equation}
This result may also be proved directly using the explicit formula. (See \cite{Gallagher1980} and \cite{SaffariVaughan1977}.)

\begin{theorem}[Selberg] \label{theorem-3}
Assuming the Riemann Hypothesis, we have for $H > 0$ 
\begin{equation} \label{Thm3estimate}
S (x, H)
 = \sum_{\substack{p_{n + 1} \leq x \\ p_{n + 1} - p_{n} \geq H}} (p_{n + 1} - p_{n})
 \ll \frac{x}{H} \log^{2} x.
\end{equation}
\end{theorem}

\begin{proof} Taking $k = 1$ in Lemma \ref{lemma-1}, we obtain Theorem \ref{theorem-3} with the additional  condition that $H \ll x^{3 / 4}$. From \eqref{Thm3estimate},
\begin{equation*}
N (x, H)
 \leq \frac{1}{H} S (x, H)
 \ll \frac{x}{H^{2}} \log^{2} x.
\end{equation*}
Now, if $H \geq C x^{1 / 2} \log x$, we can take $C$ sufficiently large to obtain  $N (x, H) < 1$. Therefore, for a sufficiently large constant,
\begin{equation} \label{CramerNS}
N (x, H)
 = S (x, H)
 = 0, \qquad \text{if} \ H \geq C x^{1 / 2} \log x.
\end{equation}
Thus, we may drop the condition $H \ll x^{3 / 4}$ in Theorem \ref{theorem-3}, since the better estimate \eqref{CramerNS} holds in this range.\footnote{Recent work \cite{CarneiroMilinovichSoundararajan2017} has determined that $C= 0.84$ is acceptable.}                                          
\end{proof}

The result \eqref{CramerNS} implies the following result of Cram\'er \cite{Cramer20} from 1920. 
\begin{corollary}[Cram\'er] Assuming the Riemann Hypothesis,
\begin{equation*}
p_{n + 1} - p_{n}
 \ll \sqrt{p_{n}} \log p_{n}.
\end{equation*}
\end{corollary}

We also have
\begin{corollary}[Selberg] Assuming the Riemann Hypothesis, 
\begin{equation*} \mathcal{C} (x)
 = \sum_{p_{n + 1} \leq x} (p_{n + 1} - p_{n})^{2}
 \ll x \log^{3} x.
\end{equation*}
\end{corollary}

\begin{proof} Since 
\begin{equation} \label{C-sum}
\mathcal{C} (x)
 = \int_{0}^{x} S (x, H) \,dH
 \leq x + \int_{1}^{x} S (x, H) \,dH,
\end{equation}
the result follows from \eqref{Thm3estimate}.
\end{proof}

\section{Selberg's Result on $\mathcal{S} (x)$}

There is a further result from Selberg's original paper that deserves special mention. 
\begin{theorem}[Selberg] \label{theorem-4}
Assuming the Riemann Hypothesis, we have
\begin{equation*}
\mathcal{S} (x)
 = \sum_{p_{n + 1} \leq x} \frac{(p_{n + 1} - p_{n})^{2}}{p_{n + 1}}
 \ll \log^{3} x.
\end{equation*}
\end{theorem}
This is only a single power of $\log x$ larger than the conjectured size of $\mathcal{S} (x)$, while the result for $\mathcal{C} (x)$ obtained in the previous section is two powers of $\log x$ larger than the conjecture. (See \eqref{expect} in Section 8.) We will see in the next section that, assuming a pair correlation conjecture for zeros of the Riemann zeta-function we can recover a $\log x$ in the results of Section 8. However, this is not true for Theorem \ref{theorem-4}, where assuming a pair correlation conjecture does not give any improvement.

From this theorem, we easily obtain the following corollary which is partly in the direction of Corollary \ref{corollary-4} proved in the next section assuming a pair correlation conjecture. 
\begin{corollary} Assuming the Riemann Hypothesis,
\begin{equation}
\liminf_{x \to \infty} \frac{\mathcal{C} (x)}{x \log^{2} x}
 \ll 1.
\end{equation}
\end{corollary}

\begin{proof} From the identity
\begin{equation} \label{Cidentity}
\mathcal{S} (x)
 = \frac{\mathcal{C} (x)}{x} + \int_{1}^{x} \frac{\mathcal{C} (u)}{u^{2}} \,du,
\end{equation} 
we see that, if $\mathcal{C} (u)\geq C u \log^{2} u$ for $\sqrt{x} \leq u \leq x$, 
\begin{equation*}
\mathcal{S} (x)
 > \int_{\sqrt{x}}^{x} \frac{\mathcal{C} (u)}{u^{2}} \,du
 \geq  C \int_{\sqrt{x}}^{x} \frac{\log^{2} u}{u} \,du
 =\frac{7}{24} C \log^{3} x,
\end{equation*}
which contradicts Theorem \ref{theorem-4} if $C$ is large enough. 
\end{proof}

\begin{proof}[of Theorem \ref{theorem-4}] We return to \eqref{Selbergintegral} and follow Selberg's proof. We take $x$ large, let $1\leq H \leq x^{3 / 4}$, and take $T= 2 x / H$. Then  \eqref{Selbergintegral} becomes
\begin{equation} \label{Selbergestimate}
\int_{1}^{x} \left(\vartheta \left(y + \frac{H y}{2 x}\right) - \vartheta (y) - \frac{H y}{2 x}\right)^{2} \frac{\,dy}{y^{2}}
 \ll \frac{H}{x} \log^{2} x.
\end{equation}
Suppose now that $p_{n + 1} \leq x$ and $p_{n + 1} - p_{n} \geq (H / x) p_{n + 1} $. Then, just as in Lemma \ref{lemma-1}, we have 
\begin{equation*}
\begin{split}
\int_{p_{n}}^{p_{n + 1} - (H / x) p_{n + 1}} \left(\vartheta \left(y + \frac{H y}{2 x}\right) - \vartheta (y) -\frac{H y}{2 x}\right)^{2} \frac{\,dy}{y^{2}}
 &= \int_{p_{n}}^{p_{n + 1}- (H / x) p_{n + 1}} \frac{H^{2}}{4 x^{2}} \,dy \\
 &\geq \frac{H^{2}}{8 x^{2}} (p_{n + 1} - p_{n}).
 \end{split}
 \end{equation*}
 
Hence, for $H \leq x^{3 / 4}$ we obtain the slight refinement of Theorem \ref{theorem-3} that, assuming the Riemann Hypothesis,
\begin{equation*}
\sum_{\substack{p_{n + 1} \leq x \\ p_{n + 1} - p_{n} \geq (H / x) p_{n + 1}}} (p_{n + 1} - p_{n})
 \ll \frac{x}{H} \log^{2} x.
\end{equation*}
The condition $H \leq x^{3 / 4}$ may be dropped in view of Cram{\'e}r's bound \eqref{Cramer}, and
Theorem \ref{theorem-4} now follows on integrating with respect to $H$ from 1 to $x$. 
\end{proof}

\section{Second Moment Results Assuming the Riemann Hypothesis and Pair Correlation}

The results in the last section assuming the Riemann Hypothesis have never been improved. However, in 1972 Montgomery \cite{Montgomery1972} found an additional conjecture on the vertical distribution of zeros of the Riemann zeta-function which allows us to obtain essentially the best possible second moment results. The Riemann Hypothesis states that the complex zeros of the Riemann zeta-function have their real part equal to $1 / 2$, so that a complex zero can be written as $\rho = 1 / 2 + i \gamma$, where $\gamma$ is real. If this conjecture is false, then the primes will have a dramatically more irregular behavior than we expect. However, all evidence points to the truth of the Riemann Hypothesis, but without pointing towards a method for its proof. Montgomery introduced the function, for $T \geq 2$, 
\begin{equation}
F (\alpha)
 = \frac{1}{N (T)} \sum_{0 < \gamma, \gamma' \leq T} T^{i\alpha (\gamma - \gamma')} \omega (\gamma - \gamma'),
\end{equation}
where $\omega (u) = 4 / (4 + u^{2})$, and
\begin{equation} \label{R-vM}
N (T)
 = \sum_{0 < \gamma \leq T} 1
 = \frac{T}{2 \pi} \log \frac{T}{2 \pi e} + O (\log T).
\end{equation}
 
\begin{Montgomery-theorem}
Assume the Riemann Hypothesis. For any real $\alpha$ we have $F (\alpha)$ is even, $F (\alpha) \geq 0$, and for $0 \leq \alpha \leq 1$ we have
\begin{equation} \label{MT}
F (\alpha)
 = T^{-2 \alpha} \log T (1 + o (1)) + \alpha + o (1), \qquad \text{as} \ \ T \to \infty.
\end{equation}
\end{Montgomery-theorem}

This theorem determines $F (\alpha)$ for $\abs{\alpha} \leq 1$, while for larger $\alpha$ Montgomery made the following conjecture.

\begin{Montgomery-conjecture}
We have
\begin{equation} \label{MC}
F (\alpha)
 = 1 + o (1) \qquad \text{for} \  1 < \alpha \leq M, \qquad \text{as} \ \ T \to \infty,
\end{equation}
for any fixed number $M$.
\end{Montgomery-conjecture}

The connection between this conjecture and the second moment for primes is given in the following theorem \cite{GoldstonMontgomery1987} from 1987. 

\begin{eqivalence-theorem}
Assuming the Riemann Hypothesis, then Montgomery's conjecture is equivalent to
\begin{equation} \label{equivalence}
m_{2} (x, h)
 \sim h x \log \frac{x}{h}
\end{equation}
uniformly for $1 \leq h \leq x^{1 - \epsilon}$.
\end{eqivalence-theorem}

In all except one of the applications to large gaps between primes we only need a weaker conjecture than \eqref{MC}, which we shall state as follows.

\begin{bounded-F-alpha-conjecture}
For any $\delta > 0$, we have $F (\alpha) \ll 1$ uniformly for $1 \leq \alpha \leq 2 + \delta$.
\end{bounded-F-alpha-conjecture}

\noindent Heath-Brown \cite{Heath-Brown82} proved, assuming the Riemann Hypothesis and the Bounded $F (\alpha)$ Conjecture, that for $1 \leq h \leq x^{1 / 2 + \delta}$
\begin{equation*}
m_{2} (x, h)
 \ll h x \log x.
\end{equation*}
Using Lemma \ref{lemma-2}, the same result stated above also holds when we replace $m_{2} (x, h)$ by $M_{2} (x, h)$.

Therefore, we obtain the  following results obtained in the same way as the corresponding results proved in the last section.
\begin{theorem}[Heath-Brown] Let $x \geq 2$ and $H \geq 1$. Assuming the Riemann Hypothesis and the Bounded $F (\alpha)$ Conjecture, we have
\begin{equation*}
N (x, H)
 \ll \frac{x}{H^{2}} \log x \quad \text{and} \quad
S (x, H)
 \ll \frac{x}{H} \log x.
\end{equation*}
\end{theorem}

\begin{corollary} \label{corollary-4}
Assuming the Riemann Hypothesis and the Bounded $F (\alpha)$ Conjecture, we have
\begin{equation} \label{Cramer}
p_{n + 1} - p_{n}
 \ll \sqrt{p_{n} \log p_{n}}.
\end{equation}
\end{corollary}

We also have
\begin{corollary} Assuming the Riemann Hypothesis and the Bounded $F (\alpha)$ Conjecture, we have
\begin{equation*}
\mathcal{C} (x)
 = \sum_{p_{n + 1} \leq x} (p_{n + 1} - p_{n})^{2}
 \ll x \log^{2} x.
\end{equation*}
\end{corollary}

Montgomery's Conjecture and the Riemann Hypothesis also give the slightly stronger result that
\begin{equation} \label{bestgap}
p_{n + 1} - p_{n}
 = o (\sqrt{p_{n} \log p_{n}}).
\end{equation}
This was first proved in \cite{GoldstonHeath-Brown1984}. It is also an easy consequence of \eqref{equivalence}, which implies, for $y = y (x) = o (x)$, that
\begin{equation*}
M_{2} (x + y, h) - M_{2} (x, h)
 = o (h x \log x).
\end{equation*}

\section{Montgomery and Soundararajan's Higher Moment Results Using a Strong Prime $k$-Tuples  Conjecture}

In place of $\pi (x; \mathcal{H}_{k})$, we now make use of
\begin{equation}
\psi (x; \mathcal{H}_{k})
 = \sum_{n \leq x} \Lambda (n + h_{1}) \cdots \Lambda (n + h_{k}),
\end{equation}
which has the advantage of counting primes with a constant density of one. A strong form of the Hardy--Littlewood Conjecture now takes the following form.

\begin{strong-Hardy-Littlewood-conjecture}
For a fixed integer $k \geq 2$ and admissible set $\mathcal{H}_{k}$, we have for any $\epsilon > 0$ and $x$ sufficiently large
\begin{equation*}
\psi (x; \mathcal{H}_{k})
 = \mathfrak{S} (\mathcal{H}_{k}) x + O (x^{1 / 2 + \epsilon}),
\end{equation*}
uniformly for $\mathcal{H}_{k} \subset [1, h]$.
\end{strong-Hardy-Littlewood-conjecture}

The following theorem is a special case of the main theorem in \cite{MontgomerySoundararajan2004}.

\begin{theorem}[Montgomery--Soundararajan] \label{theorem-6}
Suppose the Strong Hardy--Littlewood  Prime $k$-Tuples Conjecture holds for $2 \leq k \leq 2 K$ and uniformly for $\mathcal{H}_{k} \subset [1, h]$. Then for $x \geq 2$ and $\log x \leq h \leq x^{1 / 2 K}$, 
\begin{equation*}
\begin{split}
m_{2 K} (x, h)
 &= h^{K} x \log^{K} \frac{x}{h} \left(1 + o (1) + O \left(\left(\frac{h}{\log x}\right)^{-1/ 16 K}\right) \right) + O (h^{2 K} x^{1 / 2 + \epsilon}).
\end{split}
\end{equation*}
\end{theorem} 

If $K = 1$ we recover the asymptotic formula for $m_{2} (x, h)$ in \eqref{equivalence} for a restricted range of $h$. By Lemma \ref{lemma-2}, we see Theorem \ref{theorem-6} also holds for $M_{2 K} (x, h)$. Hence, we obtain the bound
\begin{equation*}
M_{2 K} (x,h)
 \ll h^{K} x \log^{K} \frac{x}{h}
\end{equation*}
for $\log x \leq h \leq x^{1 / 2 K - \delta}$ for any fixed $K \geq 1$ and $\delta > 0$. Hence, by Lemma \ref{lemma-1} we obtain the following result on large gaps.

\begin{theorem} \label{theorem-7}
Suppose that the Strong Hardy--Littlewood Prime $k$-Tuples Conjecture holds for $2 \leq k \leq 2 K$ and uniformly for $\mathcal{H}_{k} \subset [1, H]$. Then for $x \geq 2$ and $\log x \leq H \leq x^{1 / 2 K - \delta}$ for any fixed $\delta > 0$, we have 
\begin{equation*}
S (x, H)
 \ll \frac{x}{H^{K}} \log^{K} x.
\end{equation*}
\end{theorem}

This result is consistent with the Poisson Tail Conjecture \eqref{Hope} but, of course, much weaker. 

\section{Application of the Fourth Moment Bound to $\mathcal{C} (x)$ and $\mathcal{S} (x)$}

We expect that 
\begin{equation} \label{expect}
\mathcal{C} (x)
 \sim 2 x \log x \quad \text{and} \quad
\mathcal{S} (x)
 \sim \log^{2} x \quad \text{as} \ \ x \to \infty. 
\end{equation}
We see from \eqref{Cidentity} that the asymptotic formula for $\mathcal{C} (x)$ immediately implies the asymptotic formula for $\mathcal{S} (x)$, and therefore we concentrate on $\mathcal{C} (x)$.

By combining Gallagher's Theorem with the fourth moment bound in Theorem \ref{theorem-7}, we are able to nearly evaluate $\mathcal{C} (x)$. The result we obtain is the following theorem.
\begin{theorem} \label{theorem-8}
Assume the Hardy--Littlewood Prime $k$-Tuples Conjecture as in \eqref{hlconj} and the Strong Hardy--Littlewood $k$-Tuples  Conjecture for $2 \leq k \leq 4$ uniformly for $\mathcal{H}_{k} \subset [1, x^{1 / 4 - \delta} ]$ for any fixed $\delta > 0$. We have
\begin{equation*}
\mathcal{C} (x)
 =  2 x \log x (1 + o (1))+ O \Bigg(\sum_{\substack{p_{n + 1} \leq x \\ p_{n + 1} - p_{n} \geq x^{1 / 4 - \delta}}} (p_{n+1} - p_{n})^{2}\Bigg).
\end{equation*}
\end{theorem}

We do not expect any prime gaps as large as those in the error term here. However, the existence of a single prime gap of size  $\sqrt{p_{n} \log p_{n}}$ in $[x / 2, x]$ would invalidate the asymptotic formula for $\mathcal{C} (x)$. The Strong Hardy--Littlewood Conjecture without some additional information on how the error terms average when combined can not disprove the existence of such long gaps. As we have seen in \eqref{bestgap}, the pair correlation conjecture can (barely) show such gaps do not exist, but that conjecture is equivalent to a second moment results on primes in an extended range.

\begin{proof}[of Theorem \ref{theorem-8}] As in \eqref{C-sum}, we have
\begin{equation*}
\begin{split}
\mathcal{C} (x)
 &= \int_{0}^{x} S (x, H) \,dH \\
 &= \left(\int_{0}^{\lambda_{0} \log x} + \int_{\lambda_{0} \log x}^{\lambda_{1} \log x} + \int_{\lambda_{1} \log x}^{x^{1 / 4 - \delta}} + \int_{x^{1 / 4 - \delta}}^{x}\right) S (x, H) \,dH \\
 &= I_{1} +I_{2} + I_{3} + I_{4}.
\end{split}
\end{equation*}
Here, we let $\lambda_{0} \to 0$ and $\lambda_{1} \to \infty$ sufficiently slowly. Since $S (x, H) \leq x$,
\begin{equation*}
I_{1}
 \leq \lambda_{0} x \log x
 = o (x \log x).
\end{equation*}
By Theorem \ref{theorem-2},
\begin{equation*}
\begin{split}
I_{2}
 &= (1 + o (1)) x \log x \int_{\lambda_{0}}^{\lambda_{1}} (1 + \lambda) e^{-\lambda} \,d\lambda \\
 &= (1 + o (1)) x \log x \left(e^{-\lambda_{0}} (2 + \lambda_{0}) - e^{-\lambda_{1}} (2 + \lambda_{1})\right) \\
 &= (1 + o (1)) 2 x \log x.
\end{split}
\end{equation*}
Applying Theorem \ref{theorem-7} with $K = 2$, we obtain 
\begin{equation*}
I_{3}
 \ll \int_{\lambda_{1} \log x}^{x^{1 / 4 - \delta}} \frac{x \log^{2} x}{H^{2}} \,dH
 \ll \frac{x \log x}{\lambda_{1}}
 = o (x \log x).
\end{equation*}
Finally,
\begin{equation*}
I_{4}
 = \sum_{\substack{p_{n + 1} \leq x \\ p_{n + 1} - p_{n} \geq x^{1 / 4 - \delta}}} (p_{n + 1} - p_{n})\int_{x^{1 / 4 - \delta}}^{p_{n + 1} - p_{n}} \,dH
 \leq \sum_{\substack{p_{n + 1} \leq x \\ p_{n + 1} - p_{n} \geq x^{1 / 4 - \delta}}} (p_{n + 1} - p_{n})^{2}.
\end{equation*}
This completes the proof of Theorem \ref{theorem-8}. 
\end{proof}
\bigskip

\section{Some Numerical Results on Large Gaps}

In this section we present some numerical studies related to the behaviors addressed in this paper. It is instructive first to recall how the largest gap between primes no greater than $x$ increases with $x$. Figure \ref{FigSupg} is a plot of the maximal gap 
\begin{equation*}
g^{*} (x)
 = \sup_{p_{n + 1} \leq x}{\left(p_{n + 1} - p_{n}\right)},
\end{equation*}
along with the analytical asymptotic form 
\begin{equation*}
\widetilde{g}^{*} (x) = \log^{2} x
\end{equation*}
advanced by Cram{\'e}r \cite{Cramer36} over a representative sampling of approximately logarithmically spaced prime $x$.

\begin{figure}[!ht]
\begin{center}
\includegraphics[width=0.89\linewidth]{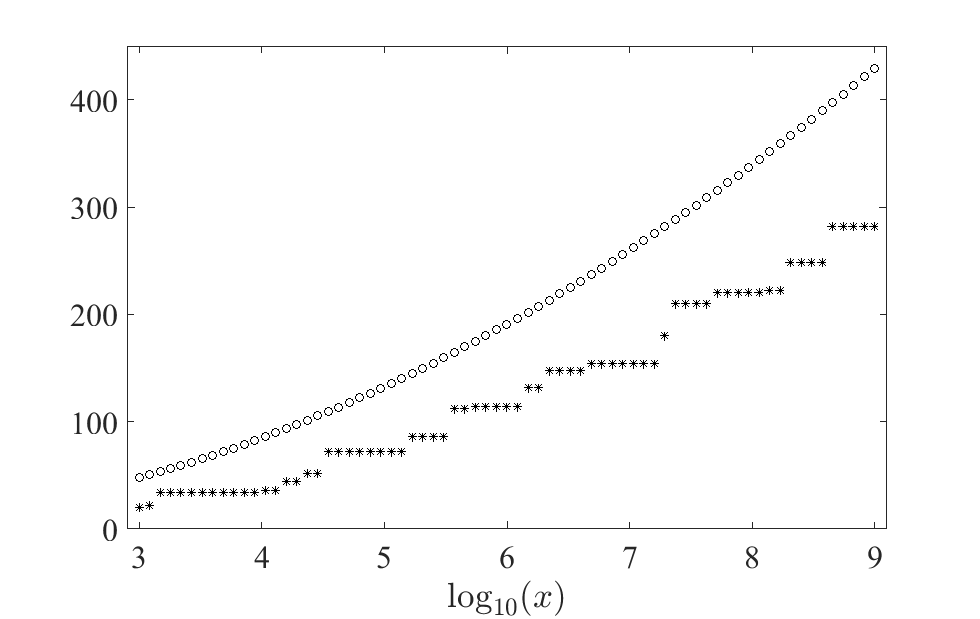}
\caption{ $g^{*} (x)$ and $\gwt^{*} (x)$ plotted in asterisks and circles, respectively, for a representative sampling of prime $x$. 
\label{FigSupg}}
\end{center}
\end{figure}

Next, we consider the large-gap counting function $N(x,H)$. For convenience, let us define the expected asymptotic form in \eqref{Npoisson} as $\widetilde{N}(x, \lambda \log x) = e^{-\lambda} x / \log x$.
Figures \ref{FigNlam1}, \ref{FigNlam3}, and \ref{FigNlam6} are logarithmic plots of $N (x, \lambda\log x)$ along with $\widetilde{N} (x, \lambda \log x)$ for $\lambda = 1, 3$, and $6$, respectively.

\begin{figure}[!ht]
\begin{center}
\includegraphics[width=0.89\linewidth]{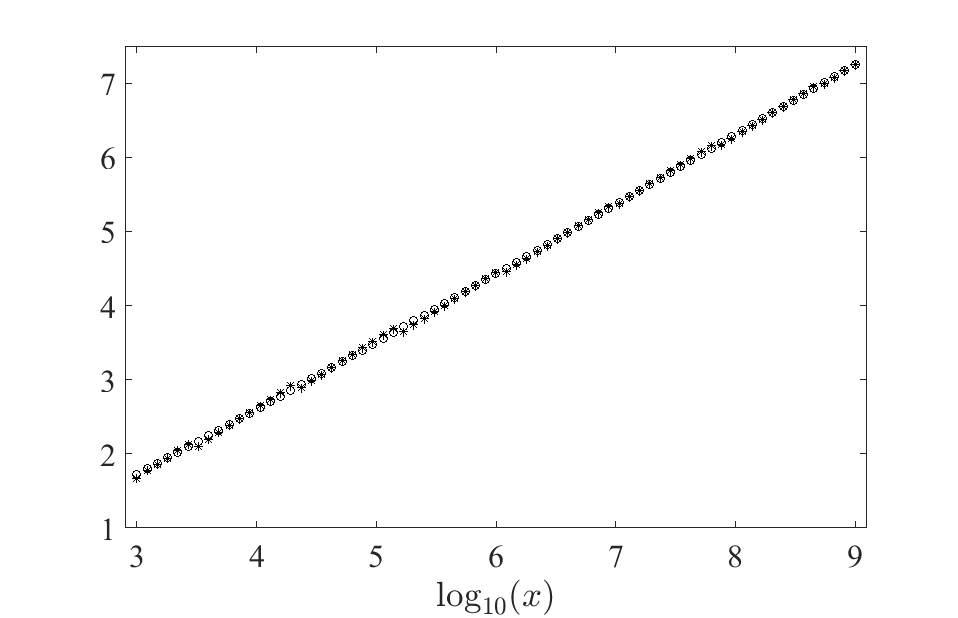}
\caption{$\log_{10} N (x,\log x)$ and $\log_{10} \widetilde{N} (x, \log x)$ plotted in asterisks and circles, respectively, for a representative sampling of prime $x$.
\label{FigNlam1}}

\vspace{15pt}

\includegraphics[width=0.89\linewidth]{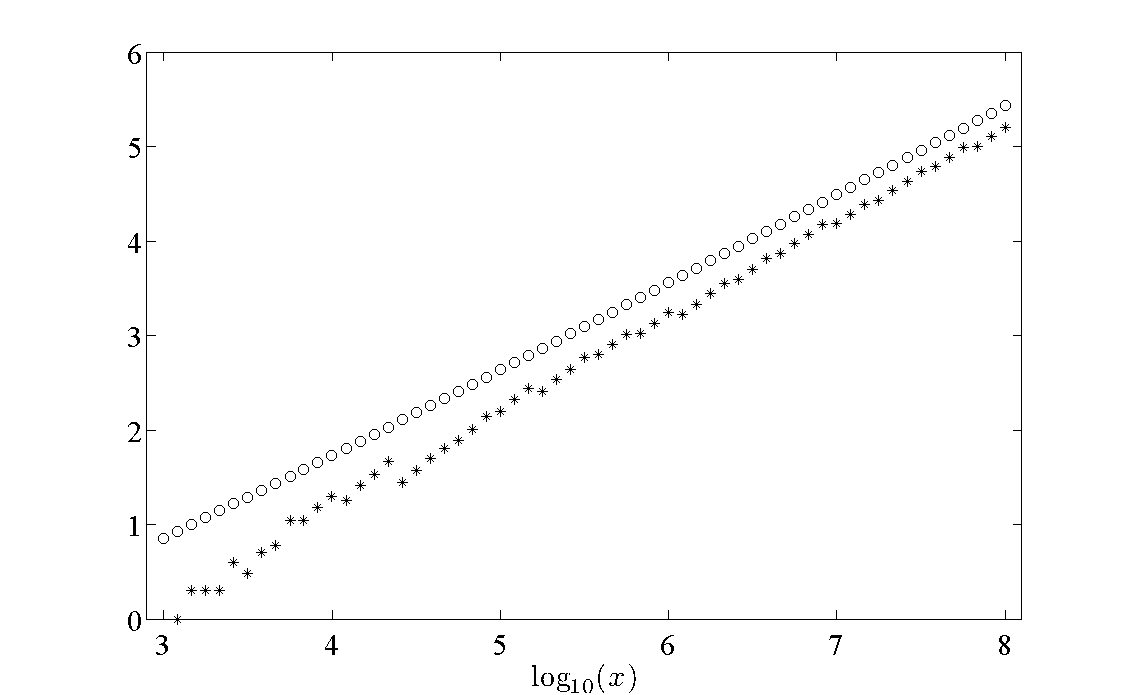}
\caption{$\log_{10} N (x, 3 \log x)$ and $\log_{10} \widetilde{N}(x, 3 \log x)$ plotted in asterisks and circles, respectively, for a representative sampling of prime $x$. 
\label{FigNlam3}}
\end{center}
\end{figure}

\begin{figure}[!ht]
\begin{center}
\includegraphics[width=0.89\linewidth]{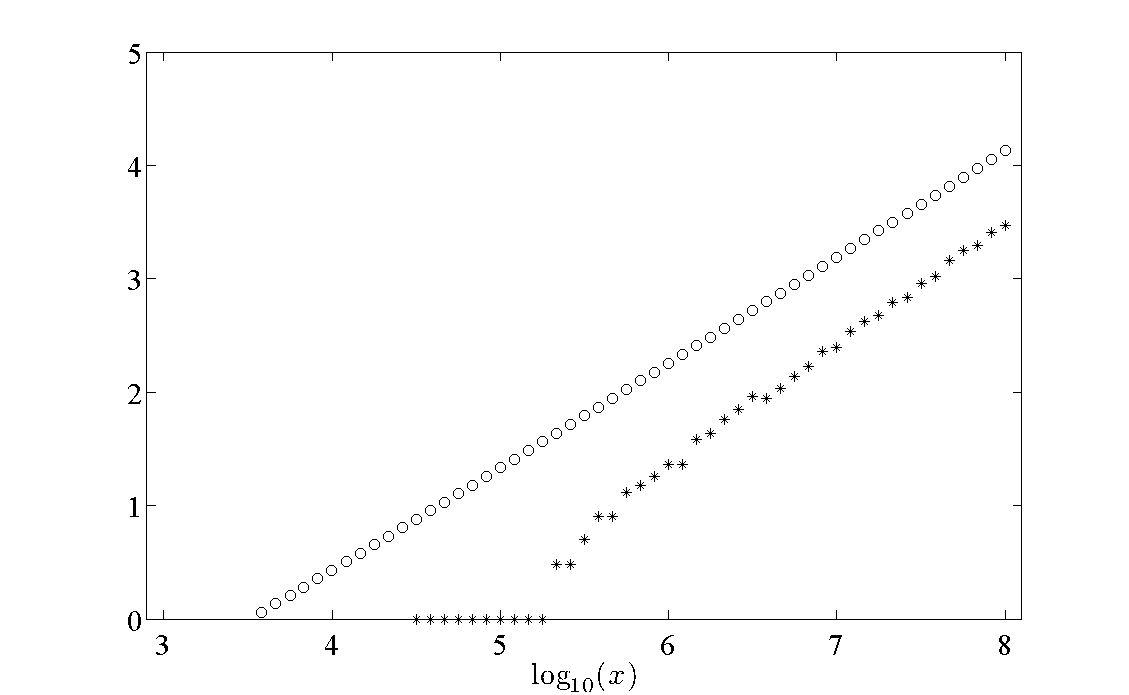}
\caption{$\log_{10} N (x, 6 \log x)$ and $\log_{10} \widetilde{N} (x, 6 \log x)$ plotted in asterisks and circles, respectively, for a representative sampling of prime $x$. 
\label{FigNlam6}}
\end{center}
\end{figure}

The weighted analogue of $N(x,H)$ is $S(x,H)$ from \eqref{S (x, H)}, and its expected asymptotic form in \eqref{Spoisson} is defined here as $\widetilde{S} (x, \lambda \log x) = (1 + \lambda) e^{-\lambda} x$ for $H = \lambda \log x$. Figures \ref{FigSlam1}, \ref{FigSlam3}, and \ref{FigSlam6} are logarithmic plots of $S(x, \lambda \log x)$ along with $\widetilde{S}(x, \lambda \log x)$ for $\lambda=1, 3$, and $6$.

\begin{figure}[!ht]
\begin{center}
\includegraphics[width=0.89\linewidth]{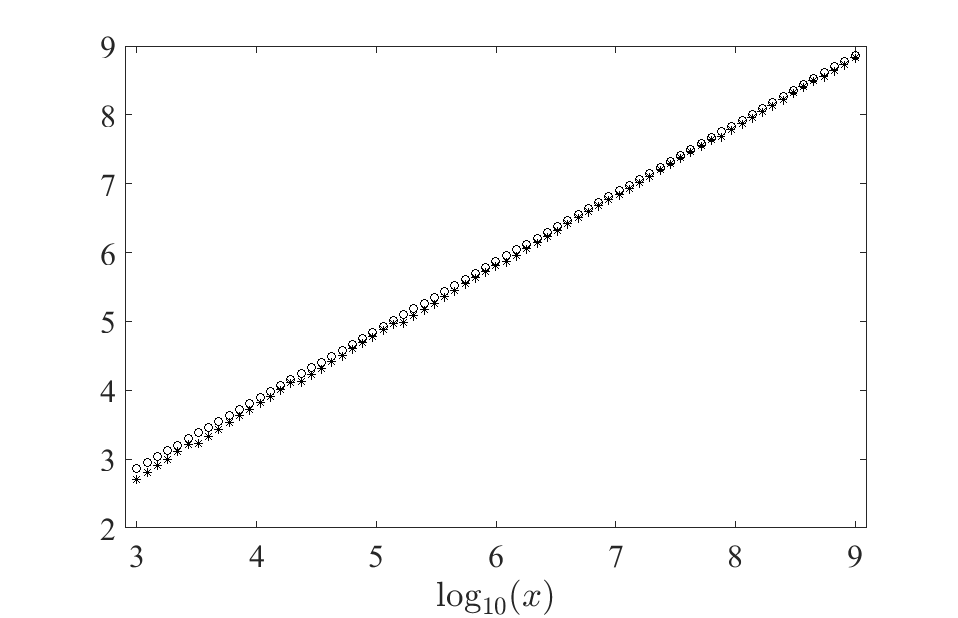}
\caption{ $\log_{10} S (x, \log x)$ and $\log_{10} \widetilde{S}(x,\log x)$ plotted in asterisks and circles, respectively, for a representative sampling of prime $x$.
\label{FigSlam1}}
\end{center}
\end{figure}

\begin{figure}[!ht]
\begin{center}
\includegraphics[width=0.89\linewidth]{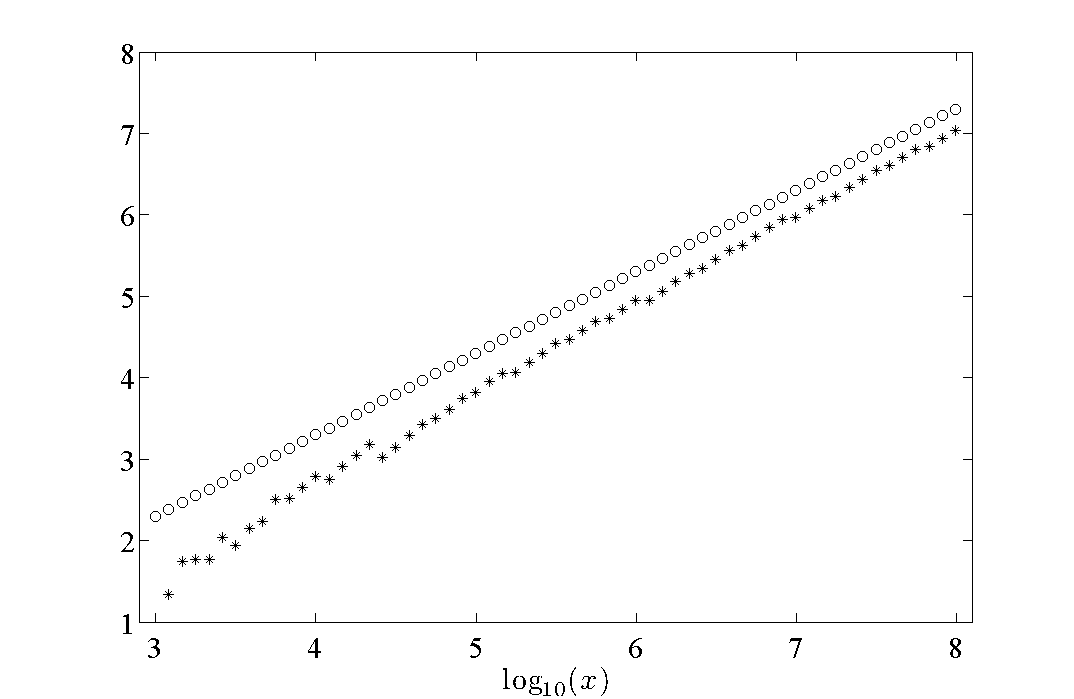}
\caption{ $\log_{10} S (x,3 \log x)$ and $\log_{10} \widetilde{S} (x, 3 \log x)$ plotted in asterisks and circles, respectively, for a representative sampling of prime $x$. 
\label{FigSlam3}}

\vspace{15pt}

\includegraphics[width=0.89\linewidth]{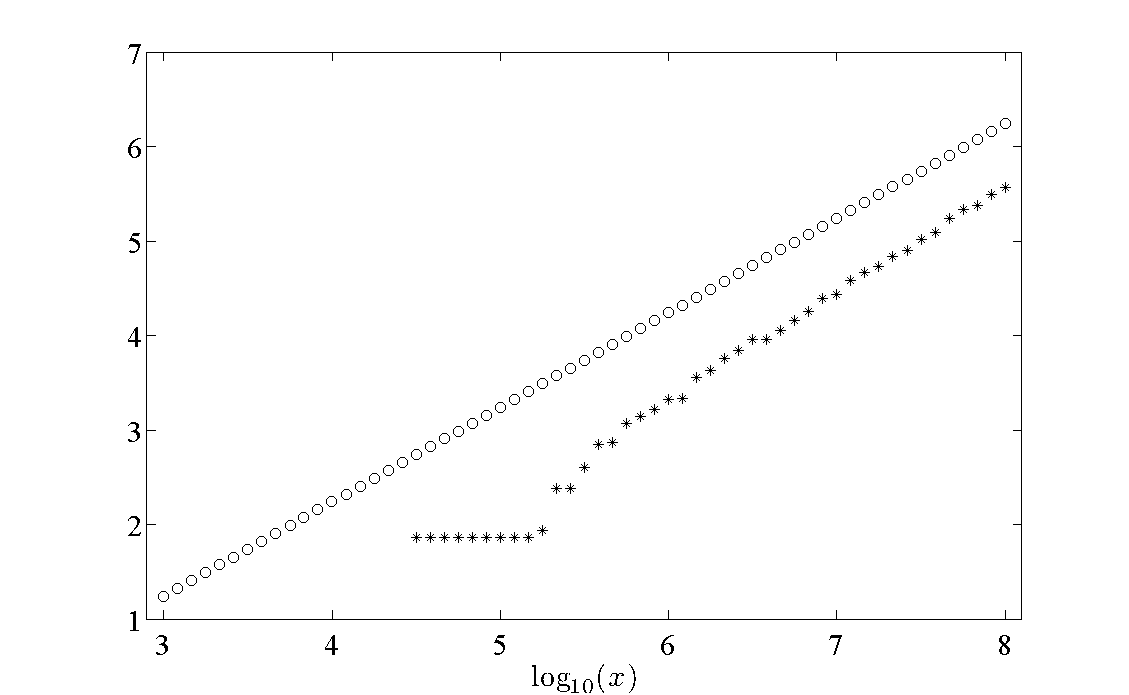}
\caption{ $\log_{10} S (x, 6 \log x)$ and $\log_{10} \widetilde{S}(x, 6 \log x)$ plotted in asterisks and circles, respectively, for a representative sampling of prime $x$. 
\label{FigSlam6}}
\end{center}
\end{figure}

The relative errors in $\widetilde{N}(x, \lambda \log x)$ and $\widetilde{S}(x, \lambda \log x)$ decrease with increasing $x$, for a given $\lambda$, in support of the conjectured asymptotic behaviors. We also find, however, that the relative error increases with increasing $\lambda$ for a fixed $x$. We may interpret this behavior as being consistent with the expected non-Poissonian properties in the distribution of large gaps.  

Finally, let us consider the terms $\mathcal{C}(x)$ and $\mathcal{S}(x)$ and the expected behaviors articulated in \eqref{expect}. For convenience, we define $\widetilde{\mathcal{C}} (x) = 2 x \log x$ to represent the expected asymptotic form of $\mathcal{C}(x)$. Curiously, the asymptotic form of $\mathcal{S}(x)$ is identical to the Cram\'er's maximal gap bound $\widetilde{g}^{*} (x)$. Figure \ref{FigCcal} is a logarithmic plot of $\mathcal{C}(x)$ along with $\widetilde{\mathcal{C}}(x)$. Figure \ref{FigScal} is a logarithmic plot of $\mathcal{S}(x)$ along with $\widetilde{g}^{*}(x)$.

\begin{figure}[!ht]
\begin{center}
\includegraphics[width=0.89\linewidth]{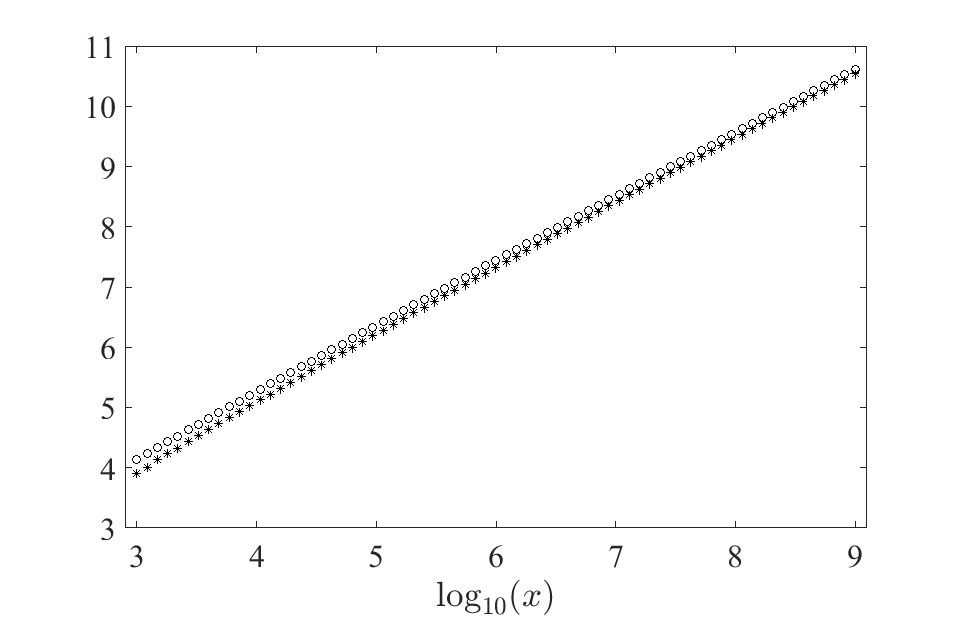}
\caption{ $\log_{10} \mathcal{C} (x)$ and $\log_{10} \widetilde{\mathcal{C}} (x)$ plotted in asterisks and circles, respectively, for a representative sampling of prime $x$. 
\label{FigCcal}}

\vspace{15pt}

\includegraphics[width=0.89\linewidth]{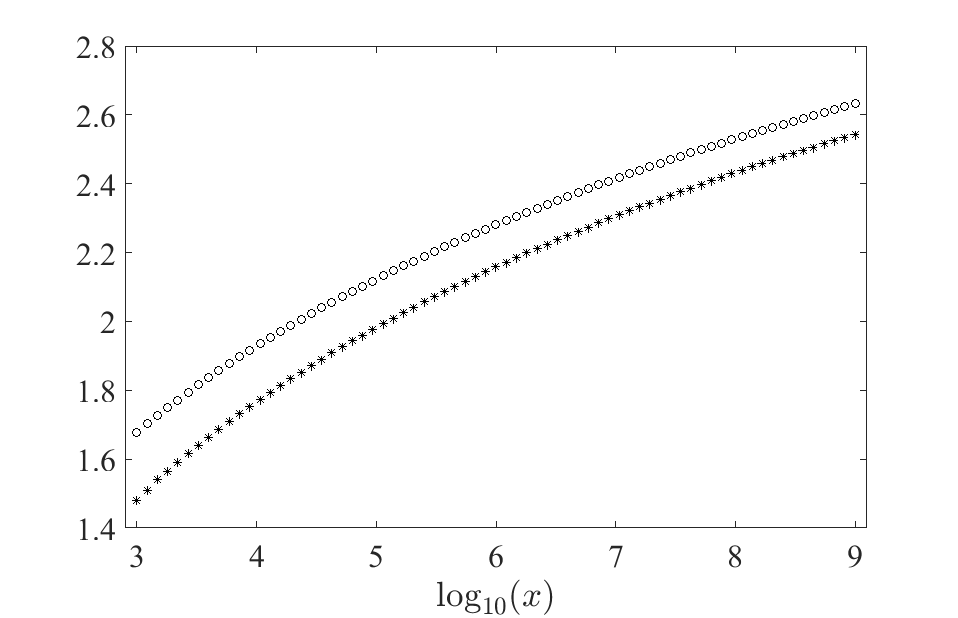}
\caption{ $\log_{10} \mathcal{S} (x)$ and $\log_{10} \widetilde{g}^{*} (x)$ plotted in asterisks and circles, respectively, for a representative sampling of prime $x$. 
\label{FigScal}}
\end{center}
\end{figure}

\end{document}